\documentclass[draft]{article}
\usepackage[dvips]{epsfig}
\input{epsf}
\usepackage{amsmath,amssymb,amsfonts}
\usepackage{amscd}
\usepackage[dvips,matrix,arrow]{xy}

\newtheorem{theorem}{Theorem}
\newtheorem{lemma}[theorem]{Lemma}

\newenvironment{proof}[1][Proof:]{\begin{trivlist}
\item[\hskip \labelsep {\bfseries #1}]}{\end{trivlist}}

\newcommand{\qed}{\nobreak \ifvmode \relax \else
      \ifdim\lastskip<1.5em \hskip-\lastskip
      \hskip1.5em plus0em minus0.5em \fi \nobreak
      \vrule height0.75em width0.5em depth0.25em\fi}

 0 3 \font\ccc =msbm10

\begin{document}

\title{On certain permutation representations of the braid group. Part II}

\author{Valentin Vankov Iliev\\
Section of Algebra,\\ Institute of
Mathematics and Informatics,\\ Bulgarian Academy of Sciences, 1113
Sofia, Bulgaria\\
e-mail: viliev@math.bas.bg}

\maketitle

\begin{abstract} In \cite{[1]}  
we find certain finite homomorphic images of Artin braid group into
appropriate symmetric groups,
which \emph{a posteriori} are extensions of the symmetric group on $n$ letters by
an abelian group. The main theorem of this paper characterizes
completely the extensions of this type that are split.

\end{abstract}


\section{Introducton}

\label{I}

This paper is a natural continuation of \cite{[1]} and we use
freely the terminology and notation used there.

In \cite[Theorem 9,(iii)]{[1]} we prove that the braid-like
permutation group $B_n(\sigma)$ is an extension of
the symmetric group $\Sigma_n$ by the abelian group
$A_n(q)$ from \cite[(3)]{[1]}, and, moreover, give a sufficient
condition for this extension to split.
Theorem~\ref{V.1.20} is the final version of \cite[Theorem 9,
(iii)]{[1]}, and proves a necessary and sufficient condition for
$B_n(\sigma)$ to be a semi-direct product of the
symmetric group $\Sigma_n$  and the abelian group $A_n(q)$.

\section{The main theorem (continued)}

\label{V}

Making use of \cite[Theorem 9,(iii)]{[1]}, we have that the group
$B_n(\sigma)$ is an extension of the symmetric
group $\Sigma_n$ by the abelian group $A_n(q)$
from \cite[(3)]{[1]}, where $q_2=\frac{q}{2}$, and the
corresponding monodromy homomorphism is defined in
\cite[Proposition 11]{[1]}.

Any (set-theoretic) section $\rho\colon\Sigma_n\to B_n(\sigma a)$
of the surjective homomorphism $\pi$ from the short exact
sequence in the proof of \cite[Theorem 9,(iii)]{[1]}
produces an element $(a_1,\hdots,
a_{n-1})$ of the direct sum $\coprod_{s=1}^{n-1}A_n(q)$ via the
rule $\rho(\theta_s)=\sigma_sa_s$, $s=1,\hdots,n-1$.

The next lemma is proved in a more general situation.
Let $A$ be a (left) $\Sigma_n$-module, and let
\[
m\colon\Sigma_n\to Aut(A),\hbox{\ }\theta_s\mapsto\iota_s,
\hbox{\ } s=1,\hdots,n-1,
\]
be the corresponding structure homomorphism. Let us set
\[
J_s=\iota_s+1,\hbox{\ }s=1,\hdots,n-1,
\]
\[
I_{r,r+1}=\iota_{r,r+1}+\iota_r+\iota_{r+1},\hbox{\
}r=1,\hdots,n-2,
\]
where $\iota_{r,r+1}=\iota_r\iota_{r+1}\iota_r=
\iota_{r+1}\iota_r\iota_{r+1}$. Note that $J_s$ and $I_{r,r+1}$
are endomorphisms of the abelian group $A$.

Let $G$ be a braid-like group, let $g_1,\hdots,g_{n-1}$ be a
generating set of $G$ that satisfies the braid relations, and let
$\gamma\colon B_n\to G$ be the corresponding surjective
homomorphism. Let $\nu\colon B_n\to \Sigma_n$ be the analogous
surjective homomorphism, corresponding to the set of
generators $\theta_1\hdots,\theta_{n-1}$ of the symmetric group
$\Sigma_n$. Moreover, let $G$ be an extension of $\Sigma_n$ by a
$\Sigma_n$-module $A$,
\begin{equation}
\begin{CD}
0@> >>A@>i>> G@>\pi>> \Sigma_n @> >>1,\\
\end{CD}\label{V.1.1}
\end{equation}
where $\nu=\pi\circ\gamma$. We identify $A$ with
the normal abelian subgroup $i(A)\leq G$ and for any $s=1,\hdots,n-1$
denote the elements $g_s^2\in A$ by $f_s$.

\begin{lemma}\label{V.1.5} The extension~(\ref{V.1.1}) splits
if and only if the linear system
\begin{equation}
\left\{
\begin{array}{ll}
J_s(a_s)=-f_s\\
s=1,\hdots,n-1\\
I_{r,r+1}(a_r-a_{r+1})=0\\
r=1,\hdots,n-2,
 \end{array}
\right.\label{V.1.6}
\end{equation}
has a solution $a_1,\hdots, a_{n-1}\in A$.

\end{lemma}

\begin{proof} The extension~(\ref{V.1.1}) splits if and only
if there exists a section $\rho$ of $\pi$, which is a homomorphism
of groups. By \cite[Theorem 4.1]{[25]}, this comes to showing that
there exist $a_1,\hdots, a_{n-1}\in A$, such that
the elements $g_sa_s$, $s=1,\hdots,n-1$, of the group
$G$ are involutions, and satisfy the braid relations.
The equalities $(g_sa_s)^2=1$ are equivalent to
\begin{equation}
g_s^2\iota_s(a_s)a_s=1. \label{V.1.10}
\end{equation}
The equalities
$g_ra_rg_{r+1}a_{r+1}g_ra_r=g_{r+1}a_{r+1}g_ra_rg_{r+1}a_{r+1}$
are equivalent to
\[
g_rg_{r+1}g_r(\iota_r\iota_{r+1})(a_r)\iota_r(a_{r+1})a_r=
g_{r+1}g_rg_{r+1}(\iota_{r+1}\iota_r)(a_{r+1})\iota_{r+1}(a_r)a_{r+1}.
\]
Since $g_s$, $s=1,\hdots,n-1$, satisfy the braid relations, we have
\begin{equation}
(\iota_r\iota_{r+1})(a_r)\iota_r(a_{r+1})a_r=
(\iota_{r+1}\iota_r)(a_{r+1})\iota_{r+1}(a_r)a_{r+1}.\label{V.1.15}
\end{equation}
Writing~(\ref{V.1.10}) and~(\ref{V.1.15}) additively,
we have that the elements
$a_1,\hdots, a_{n-1}\in A$ satisfy the linear system
\begin{equation*}
\left\{
\begin{array}{ll}
J_s(a_s)=-f_s,\\
s=1,\hdots,n-1\\
(\iota_r\iota_{r+1})(a_r)+\iota_r(a_{r+1})+a_r=
(\iota_{r+1}\iota_r)(a_{r+1})+\iota_{r+1}(a_r)+a_{r+1},\\
r=1,\hdots,n-2.
 \end{array}
\right.
\end{equation*}
After substituting $a_r$ and $a_{r+1}$ from the first group
equations into the second one for $r=1,\hdots,n-2$, we obtain the
linear system~(\ref{V.1.6}).

\end{proof}

Now, we introduce some notation. Let
\[
\hbox{\ccc Z}^n\to (\hbox{\ccc Z}/(q))^n,\hbox{\ } a\mapsto\bar{a},
\]
be the canonical surjective homomorphism of abelian groups. Note that
the elements $\bar{f}_1$,$\hdots$,$\bar{f}_{n-1}$, $\bar{f}_n$,
form a basis for the $\hbox{\ccc Z}/(q)$-module $(\hbox{\ccc Z}/(q))^n$.
The elements
$\bar{f}_1$,$\hdots$,$\bar{f}_{n-1}$, $\bar{g}_1$,
$\hdots$, $\bar{g}_{n-2}$, generate $A_n(\sigma)$
as a subgroup of $(\hbox{\ccc Z}/(q))^n$, and
the group $A_n(\sigma)$ contains the elements $\bar{h}_1$,$\hdots$,$\bar{h}_n$.
Moreover, there exists an isomorphism of abelian groups
\begin{equation*}
A_n(\sigma)\simeq\hbox{\ccc
Z}\bar{f}_1/(q)\bar{f}_1\coprod\cdots\coprod \hbox{\ccc
Z}\bar{f}_{n-1}/(q)\bar{f}_{n-1}\coprod \hbox{\ccc
Z}\bar{h}_n/(q_2)\bar{h}_n.
\end{equation*}
On the other hand, the group $(\hbox{\ccc Z}/(q))^n$ has a
structure of $\Sigma_n$-module, obtained via the isomorphism
\[
(\hbox{\ccc Z}/(q))^n\to\langle\tau\rangle^{\left(n\right)},
\]
and $A_n(\sigma)$ is its $\Sigma_n$-submodule. Throughout the end
of the paper we consider $J_s$ and $I_{r,r+1}$ as endomorphisms of
the $\hbox{\ccc Z}/(q)$-module $(\hbox{\ccc Z}/(q))^n$.

\begin{theorem} \label{V.1.20} The extension from \cite[Theorem 9,
(iii)]{[1]} splits if and only if $4$ does not divide $q$.

\end{theorem}

\begin{proof} We consider several cases.

\item Case 1. $q$ is an odd number.
We set $a_s=-\frac{1}{2}\bar{f}_s$, $s=1,\hdots,n-1$, where $\frac{1}{2}$ is taken
modulo $q$. Then the first $n-1$ equations of~(\ref{V.1.6}) are satisfied.
Further, for any $r=1,\hdots,n-2$, we have
\[
I_{r,r+1}(a_r-a_{r+1})=-\frac{1}{2}I_{r,r+1}(\bar{f}_r-\bar{f}_{r+1}),
\]
and
\[
I_{r,r+1}(\bar{f}_r)=I_{r,r+1}(\bar{f}_{r+1})=\bar{f}_r+\bar{f}_{r+1}+\bar{g}_r.
\]
Therefore, the last $n-2$ equations of~(\ref{V.1.6}) are
satisfied, too. This solution of the linear system~(\ref{V.1.6}) is
a particular case of the solutions, used in the proof of
\cite[Theorem 9, (iii)]{[1]}.

\item Case 2. $q\equiv 2\pmod{4}$.

Then $q=2q_2$, where $q_2$ is an odd number. Let us consider the
linear system~(\ref{V.1.6}) over the $\hbox{\ccc Z}/(q)$-module
$(\hbox{\ccc Z}/(q))^n$. We are looking for solutions of
~(\ref{V.1.6}) in the abelian group $A_n(\sigma)$ considered as a $\hbox{\ccc
Z}/(q)$-submodule of $(\hbox{\ccc Z}/(q))^n$.
The reduction of $A_n(\sigma)$ modulo $q_2$ is the $\hbox{\ccc
Z}/(q_2)$-module $(\hbox{\ccc Z}/(q_2))^n$, and in accord with
Case 1, the reduction of the linear system~(\ref{V.1.6}) modulo
$q_2$ is consistent. Now, let us reduce the system~(\ref{V.1.6})
modulo $2$. The reduction of $A_n(\sigma)$ modulo $2$ is the
$\hbox{\ccc Z}/(2)$-linear space $M=\hbox{\ccc Z}/(2))^{n-1}$, and
\begin{equation*}
\left\{
\begin{array}{ll}
a_r=\sum_{t=r}^{n-2}\bar{f}_t\\
r=1,\hdots,n-3\\
a_{n-2}=\bar{f}_{n-1}\\
a_{n-1}=\bar{f}_{n-2}.
\end{array}
\right.
\end{equation*}
is a solution of the linear system~(\ref{V.1.6}) in $M$. Therefore
the reduction of the linear system~(\ref{V.1.6}) modulo $2$ is
consistent over $A_n(\sigma)$. Thus, the linear
system~(\ref{V.1.6}) over the $\hbox{\ccc Z}/(q)$-module
$(\hbox{\ccc Z}/(q))^n$ is consistent, too.

\item Case 3. $q\equiv 0\pmod{4}$.

Let $a=\sum_{t=1}^{n-1}x_t\bar{f}_t+y\bar{h}_n$ be a generic
element of the abelian group $A_n(\sigma)$. We have
\[
J_1(a)=(2x_1+x_2)\bar{f}_1+2\sum_{t=3}^{n-1}[(-1)^{t-1}x_2+2x_t]f_t+
[(-1)^{n-1}x_2+2y]\bar{h}_n,
\]
and the equality $J_1(a)=-f_1$ yields $2x_1+x_2+1\equiv 0\pmod{q}$, and
$2(-1)^{n-1}x_2+4y\equiv 0\pmod{q}$, which is a contradiction.

\end{proof}

\end{document}